\documentclass[11pt]{amsart}

\usepackage{amsfonts,amssymb,graphicx,epsfig,color,fancyhdr,ifthen,mathtools}
\usepackage[letterpaper,left=1.25in,right=1.25in,top=1.5in,bottom=1.  25in]{geometry}

\usepackage{xy}

\theoremstyle{plain}
\newtheorem{thm}{Theorem}[section]
\newtheorem*{thm*}{Theorem}
\newtheorem{prop}[thm]{Proposition}
\newtheorem{cor}[thm]{Corollary}
\newtheorem{lem}[thm]{Lemma}

\newtheorem{defn}[thm]{Definition}
\theoremstyle{definition}
\newtheorem{rmk}[thm]{Remark}
\newtheorem{example}[thm]{Example}
\newtheorem{question}[thm]{Question}

\newcommand{\PK}{\mathbb{CP}^k}
\newcommand{\PP}{\mathbb{CP}^2}

\newcommand{\supp}{{\rm supp}}

\newcommand{\field}[1]{\mathbb{#1}}
\newcommand{\CC}{\ensuremath{\field{C}}} 
\newcommand{\RR}{\ensuremath{\field{R}}} 
 
\newcommand{\QQ}{\ensuremath{\field{Q}}} 
\newcommand{\DD}{\ensuremath{\field{D}}} 
\newcommand{\ZZ}{\ensuremath{\field{Z}}}

\newcommand{\Po}{\ensuremath{\field{P}^1}} 
\newcommand{\Pt}{\ensuremath{\field{P}^2}}
\newcommand{\Pk}{\ensuremath{\field{P}^k}}

\newcommand{\eps}{\epsilon}
\newcommand{\Henon}{H\'{e}non }

\begin{document}

\title{Topology of Fatou Components for Endomorphisms of $\mathbb{CP}^k$: linking with the Green's Current}

\author[S.L. ~Hruska]{Suzanne Lynch Hruska$^1$}
\address{Department of Mathematical Sciences\\
University of Wisconsin Milwaukee\\
PO Box 413\\
Milwaukee, WI 53201\\
USA}
\email{shruska@uwm.edu}

\author[R. ~Roeder]{Roland K.\ W.\ Roeder$^2$}
\address{IUPUI Department of Mathematical Sciences\\
LD Building, Room 270\\
402 North Blackford Street\\
Indianapolis, Indiana 46202-3216\\ USA}
\email{rroeder@math.iupui.edu}

\date{\today}

\begin{abstract}

Little is known about the global topology of the Fatou set $U(f)$ for
holomorphic endomorphisms $f: \mathbb{CP}^k \rightarrow \mathbb{CP}^k$, when $k
>1$.  Classical theory describes $U(f)$ as the complement in $ \mathbb{CP}^k$
of the support of a dynamically-defined closed positive $(1,1)$ current.  Given
any closed positive $(1,1)$ current $S$ on $ \mathbb{CP}^k$, we give a
definition of linking number between closed loops in $\mathbb{CP}^k \setminus
\supp S$ and the current $S$.  It has the property that if $lk(\gamma,S) \neq
0$, then $\gamma$ represents a non-trivial homology element in $H_1(
\mathbb{CP}^k \setminus \supp S)$.

As an application, we use these linking numbers to establish that many classes
of endomorphisms of $\mathbb{CP}^2$ have Fatou components with infinitely
generated first homology.  For example, we prove that the Fatou set has
infinitely generated first homology for any polynomial endomorphism of
$\mathbb{CP}^2$ for which the restriction to the line at infinity is hyperbolic
and has disconnected Julia set.  In addition we show that a polynomial skew product
of $\mathbb{CP}^2$ has Fatou set with infinitely generated first homology if
some vertical Julia set is disconnected.  We then conclude with a section of
concrete examples and questions for further study.  \end{abstract}

\maketitle

\markboth{\textsc{S. Hruska and R. Roeder}}
  {\textit{Fatou Components for Endomorphism of $\mathbb{CP}^k$:
  linking with the Green's Current
  }}

\footnotetext[1]{
Research supported in part by a grant from the University of Wisconsin Milwaukee.}

\footnotetext[2 ]{Research was supported in part by startup funds from the Department of Mathematical Sciences at Indiana University Purdue University Indianapolis.}

\section{Introduction}

Our primary interest in this paper is the topology of the Fatou set for
holomorphic endomorphisms of $\mathbb{CP}^k$ (written as $\Pk$ in the remainder
of the paper).  We develop a type of linking number that in many cases allows
one to conclude that a given loop in the Fatou set is homologically
non-trivial.  One motivation is to find a generalization of the fundamental
dichotomy for polynomial (or rational) maps of the Riemann sphere: the Julia
set is either connected, or has infinitely many connected components.  Further,
this type of result paves the way to an exploration of a potentially rich
algebraic structure to the dynamics on the Fatou set.

Given a holomorphic endomorphism $f:\Pk \rightarrow \Pk$, the {\em Fatou set}
$U(f)$ is the maximal open set on which the iterates $\{f^n\}$ form a normal
family.  The {\em Julia set} $J(f)$ is the complement, $J(f) = \Pk \setminus
U(f)$.    The standard theory \cite{FS2,HP2,Ueda} gives a convenient
description of these sets in terms the {\em Green's current} $T$. Specifically,
$T$ is a dynamically defined closed positive $(1,1)$ current with the property
that $J(f) = \supp(T)$.  We provide relevant background about the Green's
current in Section \ref{SEC:GREENS_CURRENT}.  Throughout this paper we assume
the degree of $f$ is at least two (i.e. that the components of a lift of $f$ to
$\CC^{k+1}$, with no common factors, have degree at least two).

Motivated by this description of the Fatou set, in Section  \ref{SEC:LINKING}
we define a linking number $lk(\gamma,S)$ between a closed loop $\gamma \subset
\Pk \setminus \supp \ S$ and a closed positive $(1,1)$ current $S$.  In
Proposition \ref{PROP:LINKING_DEPENDS_ON_HOMOLOGY} we will show that it depends
only on the homology class of $\gamma$, and that it defines a homomorphism 
\begin{eqnarray*}
lk(\cdot,S): H_1(\Pk \setminus \supp \ S) \rightarrow \RR/\ZZ.
\end{eqnarray*}
In particular, a non-trivial linking number in $\RR/\ZZ$ proves that the
homology class of $\gamma$ is non-trivial.
The techniques are based on a somewhat similar theory in \cite{ROE_NEWTON}.

This linking number can also be restricted to loops within any open $\Omega
\subset \Pk \setminus \supp S$, giving a homomorphism $lk(\cdot,S): H_1(\Omega)
\rightarrow \RR/\ZZ.$ If $\Omega$ is the basin of attraction for an attracting
periodic point of a holomorphic endomorphism $f:\Pk \rightarrow \Pk$ and $S$ is
the Green's current, we will show in Proposition \ref{PROP:RATIONAL_LINKING}
that the image of this homomorphism is contained in $\QQ/\ZZ$.  This provides a
natural setting to show that, under certain hypotheses, the Fatou set $U(f)$
has infinitely generated first homology:
\begin{thm}\label{THM:GENERAL_TECHNIQUE}
Suppose that $f:\Pk \rightarrow \Pk$ is a holomorphic endomorphism and $\Omega
\subset U(f)$ is a union of basins of attraction of attracting periodic points
for $f$.  If there are $c \in H_1(\Omega)$ with linking number $lk(c,T) \neq 0$
arbitrarily close to $0$ in $\QQ/\ZZ$, then $H_1(\Omega)$ is infinitely
generated.
\end{thm}
\noindent
(We prove Theorem \ref{THM:GENERAL_TECHNIQUE} in Section \ref{SEC:LINKING}.)
Note that the hypotheses of Theorem \ref{THM:GENERAL_TECHNIQUE} are satisfied
if there are piecewise smooth loops $\gamma \subset \Omega$ with $lk(\gamma,T)
\neq 0$ arbitrarily close to $0$ in $\QQ/\ZZ$.  In our applications, we often
find a loop $\gamma_0$ with nontrivial linking number, and then take an
appropriate sequence of iterated preimages $\gamma_n$ under $f^n$ so that
$lk(\gamma_n,T) \rightarrow 0$ in $\QQ/\ZZ$. 
\vspace{0.1in}

In order to apply this theory to specific examples, one needs a detailed
knowledge of the geometry of the Green's Current $T$.  In the second half of the paper
we consider two situations in which it can be readily applied to provide examples
of endomorphism $f$ of $\Pt$ having Fatou set $U(f)$ with infinitely
generated homology.

The first situation is for polynomial endomorphisms of $\Pt$, that is,
holomorphic maps of $\Pt$ that are obtained as the extension  a polynomial map
$f(z,w) = (p(z,w),q(z,w))$ on $\CC^2$. 
Such mappings (and their generalizations to $\Pk$) were studied in \cite{BEDFORD_JONSSON}.
Given a polynomial endomorphism $f:\Pt \rightarrow \Pt$, the line at
infinity, denoted by $\Pi$, is totally invariant and superattracting.
Therefore the restriction of $T$ to $\Pi$ can be understood using the
dynamics of the resulting rational map of $f_{|\Pi}$ and its Julia set $J_\Pi$.
In Section \ref{SEC:ENDO} we prove the following theorem.

\begin{thm}\label{THM:JPI_DISCONN}
Suppose that $f$ is a polynomial endomorphism of $\Pt$ with restriction $f_{|\Pi}$ to the line at infinity $\Pi$.  If $f_{|\Pi}$ is hyperbolic and $J_\Pi$ is disconnected, then the Fatou set $U(f)$ has
infinitely generated first homology.
\end{thm}

\noindent
This theorem provides for many examples of polynomial endomorphisms $f$ of $\Pt$ with interesting homology of $U(f)$.
We present one concrete family in Example \ref{EXAMPLE:RABBITS}.

We then consider the special family of polynomial endomorphisms known as
polynomial skew products.   While Theorem 1.2 applies to certain polynomial
skew products, we develop additional sufficient criteria for $U(f)$ to have
interesting homology. 

A polynomial skew product is a polynomial endomorphism having the form $f(z,w) = (p(z),
q(z,w))$, where $p$ and $q$ are polynomials.  We assume that ${\rm deg}(p) =
{\rm deg} (q) = d$ and $p(z) = z^d + O(z^{d-1})$ and $q(z) = w^d
+O_z(w^{d-1})$, where we have normalized leading coefficients.  Since $f$
preserves the family of vertical lines $\{z \} \times \CC$, one can analyze $f$
via the collection of one variable fiber maps $q_z(w) = q(z,w)$, for each $z
\in \CC$.  In particular, one can define fiber-wise filled Julia sets $K_z$ and
Julia sets $J_z :=\partial K_z$ with the property that $w \in  \CC \setminus K_z$ if and only if
the orbit of $(z,w)$ escapes vertically to a superattracting fixed point $[0:1:0]$ at infinity.

For this reason, polynomial skew products provide an
accessible generalization of one variable dynamics to two variables and have
been previously studied by many authors, including Jonsson in \cite{JON_SKEW}
and DeMarco, together with the first author of this paper, in \cite{S}.
In Section~\ref{SEC:SKEW} we provide the basic background on
polynomial skew products and prove:

\begin{thm}\label{THM:MAIN} 
Suppose $f(z,w) = (p(z),q(z,w))$ is a polynomial skew product.
\begin{itemize}
\item If $J_{z_0}$ is disconnected for any $z_0 \in J_p$, then
$W^s([0:1:0])$ has infinitely generated first homology.

\item Otherwise, $W^s([0:1:0])$ is homeomorphic to an open ball.
\end{itemize}
 \end{thm}

\noindent
The first statement is obtained by using Theorem \ref{THM:GENERAL_TECHNIQUE}, while the second is obtained using Morse Theory.

For any endomorphism there is also the measure of maximal
entropy $\mu = T \wedge T$.   Thus another candidate for the name ``Julia set'' is
$J_2 := \text{supp}(\mu)$. The Julia set that is defined as the
complement of the Fatou set is sometimes denoted by $J_1$, to distinguish it from $J_2$.

The condition from Theorem \ref{THM:MAIN} that for some $z_0 \in \CC$, $J_{z_0}$ is
disconnected might seem somewhat unnatural.  A seemingly more natural condition
might be that $J_2$ is disconnected, since for polynomial skew products it is
known (see \cite{JON_SKEW}) that $J_2 = \overline{\bigcup_{z \in J_p} J_z}$.
However, 
in Example \ref{EXAMPLE:J1DISCONNJ2CONN} we present certain polynomial skew products with
$J_2$ connected, but with the Fatou set having infinitely generated first
homology. (These examples are obtained by applying 
Theorem \ref{THM:MAIN} to  examples from \cite{JON_SKEW} and \cite{S}.) 
In fact, some of these examples persist over an open set within a one-variable holomorphic family of polynomial skew products.
Therefore, 
for polynomial skew products,
connectivity of the fiber Julia sets $J_z$ is at least as important as the connectivity of $J_2$ to understanding the homology of the
Fatou set.  

In Section \ref{SEC:QUADRATIC_FAMILY} we provide an example of a family of
polynomial skew products $f_a$ depending on a single complex parameter $a$ with
the following property: if $a$ is in the Mandelbrot set $\mathcal{M}$, then the Fatou
set $U(f_a)$ is homeomorphic to the union of three open balls, while if $a$ is outside of $\mathcal{M}$
then $H_1(U(f_a))$ is infinitely generated.

Since neither of the sufficient conditions from Theorems \ref{THM:JPI_DISCONN}
and \ref{THM:MAIN} extend naturally to general endomorphisms of $\Pk$, it remains a
mystery what is an appropriate condition for endomorphism to have non-simply
connected Fatou set.  We conclude Section \ref{SEC:FURTHER_APPS}, and this
paper, with a discussion of a few potential further applications of the
techniques of this paper to holomorphic endomorphisms of $\Pk$.

\subsection*{Acknowledgments}
We thank John H. Hubbard for bringing us and some central ideas together at the
start of this project.  We have benefited greatly from discussions with many
people, including Eric Bedford, Greg Buzzard, Laura DeMarco, Mattias Jonsson, Sarah Koch, 
Lex Oversteegen, Rodrigo Perez, Han Peters, Enrique Pujals and Nessim Sibony.
The second author thanks Mikhail Lyubich and Ilia Binder for their mathematical
guidance and financial support while he was a postdoctoral fellow.

We thank the anonymous referee for many helpful comments, particularly those
encouraging us to prove stronger statements in Theorems \ref{THM:MAIN} and
\ref{THM:QUADRATIC_FAMILY}.


\section{The Green's current $T$}
\label{SEC:GREENS_CURRENT}

We provide a brief reminder of the properties of the Green's
current that will be needed later in this paper.  We refer the reader who would
like to see more details to \cite{FS2,HP2,Ueda}.  While the following construction works
more generally for generic (algebraically stable) rational maps having points of
indeterminacy, we restrict our attention to globally holomorphic maps of
$\Pk$.

Suppose that $f:\Pk \rightarrow \Pk$ is holomorphic and that the Jacobian of
$f$ does not identically vanish on $\Pk$.  Then $f$ lifts to a polynomial map
$F:\CC^{k+1} \rightarrow \CC^{k+1}$ each of whose coordinates is a homogeneous
polynomial of degree $d$ and so that the coordinates do not have a common
factor.  It is a theorem that
\begin{eqnarray}\label{EQN:GREEN1}
G(z) = \lim_{n\rightarrow \infty} \frac{1}{d^n} \log ||F^n(z)||
\end{eqnarray}
\noindent
converges to a plurisubharmonic\footnote{We will often use the abbreviation PSH in place of plurisubharmonic and we use the convention that PSH functions
cannot be identically equal to $-\infty$.}
function $G:\CC^{k+1} \rightarrow [-\infty,\infty)$
called the {\em Green's function associated to $f$}.  Since $f$ is globally
well-defined on $\Pk$ we have that $F^{-1}(0) = 0$.  It has been established
that $G$ is Holder continuous and locally bounded on $\CC^{k+1} \setminus
\{0\}$.

If $\pi:\CC^{k+1} \setminus \{0\} \rightarrow  \Pk$ is the canonical
projection, there is a unique positive closed $(1,1)$ current $T$ on $\Pk$
satisfying $\pi^{*} T = \frac{1}{2\pi} dd^c G$.  (This normalization is not uniform--many authors do not divide by $2\pi$.)  More explicitly, consider any open set $V
\subset \Pk$ that is ``small enough'' so that a holomorphic section $\sigma : V
\rightarrow \CC^{k+1}$ of $\pi$ exists.  Then, on $V$ we have that $T$ is given
by $T = \frac{1}{2\pi} dd^c (G \circ \sigma)$.
    Choosing appropriate open sets covering $\Pk$ and sections of $\pi$ on each of them, the result extends to all of $\Pk$ producing a
single closed positive $(1,1)$ current on $\Pk$ independent of the choice of
open sets and sections used.  See \cite[Appendix A.4]{SI1}.
By construction, the Green's current satisfies the invariance $f^*T = d\cdot T$.  (See Section \ref{SUBSECTION:INVARIANCE_AND_RESTRICTION} for the definition of the pull-back $f^*T$.)

Recall that the Fatou set $U(f)$ is the maximal open set in $\Pk$ where the
family of iterates $\{f^n\}$ form a normal family and that the Julia set of $f$
is given by $J(f) = \Pk \setminus U(f)$.  A major motivation for studying the Green's current is the following.

\begin{thm}
Let $f:\Pk \rightarrow \Pk$ be a holomorphic endomorphism and let $T$ be the Green's current corresponding to $f$.  Then, $J(f) = \supp \ T$.
\end{thm}
\noindent
See, for example, \cite[Proposition 4.5]{FS2} or \cite[Theorem 2.2]{Ueda}.

\begin{rmk}\label{RMK:AFFINE_GREENS_FUNCTION}
If $f$ is a polynomial endomorphism, another form of Green's function, given by
\begin{eqnarray}\label{EQN:GREEN2}
G_{\rm affine}(z) = \lim \frac{1}{d^n} \log_+ ||f^n(z)||
\end{eqnarray}
\noindent
is often considered in the literature.
(Here $\log_+ = \max\{\log,0\}$.)  The result is again a PSH function $G:\CC^k \rightarrow [0,\infty)$.

We can relate $G_{\rm affine}$ to $G$ in the following way.  Consider the open
set $V=\CC^k \subset \Pk$.  Using the section
$\sigma(z_1,\cdots,z_k) = (z_1,\cdots,z_k,1)$, we find $G_{\rm
affine}(z_1,\cdots,z_k) = G \circ \sigma(z_1,\cdots,z_k)$ because $||F^k \circ
\sigma||$ only differs from $||f^k||$ by a bounded amount for each iterate
$k$.

Therefore, if $f$ is a polynomial endomorphism of $\Pk$, one can compute $T$ on $\CC^k$ using the formula $T = \frac{1}{2\pi} dd^c G_{\rm affine}$.
\end{rmk}

\begin{rmk}
Note that formulae (\ref{EQN:GREEN1}) and (\ref{EQN:GREEN2}) are independent of the norm $\|\cdot \|$ that is used since any two norms are equivalent up to a multiplicative constant.
\end{rmk}

\begin{rmk}\label{RMK:GREEN_1D}
When $k = 1$, the resulting Green's current is precisely the measure of maximal entropy $\mu_f$ whose support is the Julia set $J(f) \subset \Po$.  If $f$ is a polynomial, then $\mu_f$ also coincides with the harmonic measure on $K(f)$, taken with respect to the point at infinity. 
\end{rmk}

\section{Linking with a closed positive $(1,1)$ current in $\Pk$.}
\label{SEC:LINKING}

Suppose that $S$ is an (appropriately normalized) closed positive $(1,1)$
current on $\Pk$ and $\gamma \subset \Pk \setminus \supp(S)$ is a piecewise
smooth closed loop.  We will define a linking number $lk(\gamma,S) \in \RR /
\ZZ$, depending only on the homology class $[\gamma] \in H_1(\Pk \setminus
\supp(S))$. 

\subsection{Classical linking numbers in $\mathbb{S}^3$}
\label{SUBSEC:CLASSICAL_DEF}

Classically one considers the linking number of two oriented loops $c$ and $d$
in $\mathbb{S}^3$.  The linking number $lk(c,d) \in \mathbb{Z}$ is found by
taking any oriented surface $\Gamma$ with oriented boundary $c$ and defining
$lk(c,d)$ to be the signed intersection number of $\Gamma$ with $d$ as in
Figure \ref{LINK}.  For this and many equivalent definitions of linking number
in $\mathbb{S}^3$ see \cite[pp. 132-133]{ROLF}, \cite[pp. 229-239]{BO_TU}, and \cite[Problems
13 and 14]{MILNOR_TOP}.

\begin{figure}[!ht]\label{LINK}
\begin{center}
\begin{picture}(0,0)%
\epsfig{file=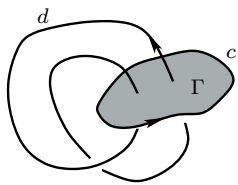}%
\end{picture}%
\setlength{\unitlength}{3947sp}%
\begingroup\makeatletter\ifx\SetFigFont\undefined%
\gdef\SetFigFont#1#2#3#4#5{%
  \reset@font\fontsize{#1}{#2pt}%
  \fontfamily{#3}\fontseries{#4}\fontshape{#5}%
  \selectfont}%
\fi\endgroup%
\begin{picture}(1577,1103)(1519,-966)
\put(2690,-393){\makebox(0,0)[lb]{\smash{{\SetFigFont{8}{9.6}{\familydefault}{\mddefault}{\updefault}{\color[rgb]{0,0,0}$\Gamma$}%
}}}}
\put(2912,-170){\makebox(0,0)[lb]{\smash{{\SetFigFont{8}{9.6}{\familydefault}{\mddefault}{\updefault}{\color[rgb]{0,0,0}$c$}%
}}}}
\put(1724, 53){\makebox(0,0)[lb]{\smash{{\SetFigFont{8}{9.6}{\familydefault}{\mddefault}{\updefault}{\color[rgb]{0,0,0}$d$}%
}}}}
\end{picture}%
\end{center}
\caption{Here  $lk(c,d) = +2$.}
\end{figure}

To see that this linking number is well-defined notice that assigning $lk(c,d) = [\Gamma]
\cdot [d]$, where $\cdot$ indicates the intersection product on
$H_*(\mathbb{S}^3,c)$, coincides with the classical definition.
(For background on the intersection product on homology, see \cite[pages 366-372]{BREDON}.)
 If $\Gamma'$
is any other 2-chain
 with $\partial \Gamma ' = c$ then $\partial (\Gamma -
\Gamma') = [c]- [c] = 0$ and $(\Gamma - \Gamma')$ represents a homology class in
$H_2(\mathbb{S}^3)$.  Since $H_2(\mathbb{S}^3) = 0$, $[\Gamma - \Gamma'] = 0$
forcing $[\Gamma - \Gamma'] \cdot [d] = 0$.  Therefore: $[\Gamma] \cdot [d] =
[\Gamma'] \cdot [d]$, so that $lk(c,d)$ is well defined.

\vspace{.1in}

\subsection{Generalization}
Given any closed positive $(1,1)$ current $S$ on $\Pk$ and any piecewise smooth two chain $\sigma$ in $\Pk$ with $\partial \sigma$ disjoint from $\supp \ S$,
we can define
\begin{eqnarray*}
\left<
\sigma,S \right> = \int_\sigma \eta_S
\end{eqnarray*}
\noindent
where $\eta_S$ is a smooth approximation of $S$ within it's cohomology class in
$\Pk-\partial \sigma$, see \cite[pages 382-385]{GH}.  The resulting number
$\left< \sigma,S \right>$ will depend only on the cohomology class of $S$ and
the homology class of $\sigma$ within $H_2(\Pk,\partial \sigma)$.  (Note that if $S$ is already a smooth form, one need not require 
that $\partial \sigma$ be disjoint from $\supp \ S$.)

Notice that $H_2(\Pk)$ is generated by the class of any complex projective line
$L \subset \Pk$.  Since $S$ is non-trivial, $\left<L,S\right> \neq 0$, so that
after an appropriate rescaling we can assume that $\left<L,S\right> =1$.  In
the remainder of the section we assume this normalization.  (It is satisfied by
the Green's Current from Section \ref{SEC:GREENS_CURRENT}.)

What made the linking numbers in $\mathbb{S}^3$ well-defined, independent of
the choice of $\Gamma$, is that $H_2(\mathbb{S}^3) = 0$.   One cannot
make the immediately analogous definition that $lk(\gamma,S) = \left<
\Gamma,S \right>$ in $\Pk$, since $H_2(\Pk) \neq 0$ implies that $\left< \Gamma,S
\right>$ can depend on the choice of $\Gamma$.  For example, given $\Gamma$
with $\partial \Gamma = \gamma$ then $\partial \Gamma' = \gamma$ for $\Gamma' =
\Gamma + L$, however $\left< \Gamma',S \right> - \left< \Gamma,S \right> =
\left<L,S \right>  = 1 \neq 0$.  

There is a simple modification: Given any $\Gamma$ and $\Gamma'$ both having
boundary $\gamma$, $[\Gamma' - \Gamma] \in H_2(\Pk)$ so that $[\Gamma' -
\Gamma] \sim k\cdot [L]$ for some $k \in \ZZ$.  Since $S$ is normalized, this
gives that $\left< \Gamma',S \right> = \left< \Gamma,S \right>  \ (\text{mod} \
1)$. 

\begin{defn}\label{DEFN:LK}
Let $S$ be a normalized closed positive $(1,1)$ current on $\Pk$ and let $\gamma$
be a piecewise smooth closed curve in $\Pk \setminus \supp(S)$.
We define the {\em linking number}
$lk(\gamma,S)$ by
\begin{eqnarray*}
lk(\gamma,S) := \left< \Gamma,S \right> \ (\text{mod} \ 1)
\end{eqnarray*}
\noindent
where $\Gamma$ is any piecewise smooth two chain with $\partial \Gamma =
\gamma$.
\end{defn}

Unlike linking numbers between closed loops in $\mathbb{S}^3$, it is often the
case that that $\left< \Gamma,S \right> \not \in \ZZ$, resulting in non-zero
linking numbers $(\text{mod} \ 1)$.  See Subsection \ref{SUBSEC:EXAMPLE_LK} for an explicit example.

\begin{prop}\label{PROP:LINKING_DEPENDS_ON_HOMOLOGY}
If $\gamma_1$ and $\gamma_2$ are homologous in $H_1(\Pk \setminus \supp \ S)$, then
$lk(\gamma_1,S) = lk(\gamma_2,S)$.
\end{prop}

\begin{proof}
Let $\Gamma$ be any piecewise smooth two chain contained in $\Pk
\setminus \supp \ S$ with $\partial \Gamma = \gamma_1 - \gamma_2$.  Then, since
$\Pk \setminus \supp \ S$ is open and $\Gamma$ is compact subset,
$\Gamma$ is bounded away from the support of $S$.  Consequently for any smooth
approximation $\eta_S$ of $S$  supported in a sufficiently small neighborhood of $S$,
we have $lk(\gamma_1,S) - lk(\gamma_2,S) = \int_\Gamma \eta_T = 0$.
\end{proof}

\begin{cor}\label{COR:NONZERO_LINK}If $\gamma \in \Pk \setminus \supp \ S$ with $lk(\gamma,S) \neq 0$, then $\gamma$ is a homologically non-trivial
loop in $\Pk \setminus \supp \ S$.
\end{cor}

Since $lk(\gamma,S)$  depends only on the homology class of $\gamma$ and the pairing $\left<\cdot,S\right>$ is linear in the space
of chains $\sigma$ (having $\partial \sigma$ disjoint from $\supp \ S$), the linking number descends to a homomorphism:
\begin{eqnarray*}
lk(\cdot,S): H_1(\Pk \setminus \supp \ S) \rightarrow \RR/\ZZ.
\end{eqnarray*}
\noindent
Similarly $lk(\cdot,S) :H_1(\Omega) \rightarrow \RR/\ZZ$ for any open $\Omega \subset \Pk \setminus \supp \ S$.

\begin{rmk}{\bf (Topological versus Geometric linking numbers.)}\label{RMK:GEOMETRIC_VS_TOPOLOICAL_LK}
The classical linking number, and also 
Definition \ref{DEFN:LK}, depend only on the homology class of the loop $\gamma$ (in
the complement of some other loop of the support of some current, respectively.)

A linking number depending on the geometry of
$\gamma$ is given by
\begin{eqnarray*}\label{EQN:ALT_LINK_DEF}
\widehat{lk}(\gamma,T) := \left<\Gamma,S - \Omega \right> \in \mathbb{R},
\end{eqnarray*}
\noindent
\noindent
where $\partial \Gamma = \gamma$ and $\Omega$ is (normalization of) the K\"ahler form defining the
Fubini-Study metric on $\Pk$.  Given any $\Gamma$ and $\Gamma'$ both having boundary
$\gamma$ we have that  $\left<\Gamma-\Gamma',T - \Omega \right> = 0$, since $S$ and $\Omega$
are cohomologous.  (In the language of \cite[p. 132]{ROE_NEWTON}, we say that
$T - \Omega$ is in the ``linking kernel of $\Pk$''.)

Because $\supp \ \Omega = \Pk$, the statement of Proposition
\ref{PROP:LINKING_DEPENDS_ON_HOMOLOGY} does not apply.
Rather, $\widehat{lk}(\gamma,S)$  depends on the geometry of $\gamma \subset
\Pk \setminus \supp \ S$.  In fact, similar linking numbers were used in \cite{HL1,HL2} to
determine if a given real-analytic $\gamma$ has the appropriate geometry to be the boundary of a positive holomorphic
$1$-chain (with bounded mass).
\end{rmk}

\begin{rmk}\label{RMK:GENERAL_MANIFOLDS}{\bf (Other manifolds.)}
Suppose that $M$ is some other compact complex manifold with $H_2(M)$ of rank
$k$, generated by $\sigma_1,\ldots,\sigma_k$.  If
$\left<\sigma_1,S\right>,\ldots,\left<\sigma_k,S\right>$ are rationally
related,  then $S$ can be appropriately rescaled so that Definition
\ref{DEFN:LK} provides a well-defined linking number between any piecewise
smooth closed curve $\gamma \in M \setminus \supp S$ and $S$.  If $H_2(M)$ has
rank $k > 1$, this provides a rather restrictive cohomological condition on
$S$.   (It is similar to the restriction of being in the ``linking kernel'' described in \cite{ROE_NEWTON}.)
\end{rmk}

\subsection{Invariance and restriction properties of  $\langle \cdot , \cdot \rangle$}
\label{SUBSECTION:INVARIANCE_AND_RESTRICTION}

Suppose that $\Omega, \Lambda$ are open subsets of $\CC^j$ and $\CC^k$, and $f:
\Omega \rightarrow \Lambda$ is a (possibly ramified) analytic mapping.  Let $S$
be a closed positive $(1,1)$ current given on $\Lambda$ by $S = dd^c u$ for
some PSH function $u$.  If $f(\Omega)$ is not contained in the polar locus of
$u$, then the {\em pull-back of $S$ under} $f$ is defined by pulling back the
potential: $f^*(S) := dd^c (u \circ f)$.  Since $u \circ f$ is not identically
equal to $-\infty$, it is also a PSH function, and $f^*(S)$ is a well-defined
closed positive $(1,1)$ current.  

Suppose that $M$ and $N$ are complex manifolds and that $S$ is a closed
positive $(1,1)$ current on $N$.  If $f:M \rightarrow N$ is a holomorphic map
with $f(M)$ not entirely contained in the polar locus of $S$, then the
pull-back $f^*S$ can be defined by taking local charts and local potentials for
$S$.  See \cite[Appendix A.7]{SI1} and \cite[p.
330-331]{HP2} for further details.

\begin{prop}\label{PROP:PAIRING_INV}
Suppose that $S$ is a closed positive $(1,1)$ current on $N$ and $f:M
\rightarrow N$, with $f(M)$ not contained in the polar locus of  $S$.
If $\sigma$ is a piecewise smooth two chain in $M$ with $\partial \sigma$
disjoint from $\supp \ f^* S$, then 
$\left< f_* \sigma, S \right> = \left< \sigma, f^*S \right>$.
\end{prop}

\begin{proof}
Since $f(M)$ is not contained in the polar locus of $S$, $f^*S$ is well-defined.
Since $\partial \sigma$ is disjoint from $\supp f^* S$, $\partial f(\sigma)$ is disjoint from $\supp S$.  
Let $\eta_S$ be a smooth approximation of $S$ in the same cohomology class as $S$
and having support disjoint from $\partial f(\sigma)$.
Then, $\left< f_* \sigma, S \right> = \int_{f_* \sigma} \eta_S = \int_\sigma f^* \eta_S
= \left< \sigma, f^*S \right>$, since $f^* \eta_S$ is a smooth approximation of $f^*S$.
\end{proof}

In the case that $M$ is an analytic submanifold of $N$ not entirely contained
in the polar locus of $S$, the restriction of $S$ to $M$ is defined by
$S\vert_M := \iota^*S$, where $\iota:M \rightarrow N$ is the inclusion.  When
computing linking numbers, we will often choose $\Gamma$ within some
one-complex dimensional curve $M$ in $N$, with $M$ not contained in the polar
locus of $S$.  In that case $S|_M$ is a positive measure on $M$ and we can use
the following:

\begin{cor}\label{COR:RESTRICT_THEN_INTEGRATE}
Let $S$ be a positive closed $(1,1)$ current on $N$ and $M$ be an analytic
curve in $N$ that is not entirely contained in the polar locus of $S$.  If
$\Gamma$ is a piecewise smooth two chain in $M$ with $\iota(\partial \Gamma)$
disjoint from $\supp \ S$, then
\begin{eqnarray}
\left<\iota(\Gamma),S \right> =\int_\Gamma S \vert_M.
\end{eqnarray}
\end{cor}

\begin{proof}
Proposition \ref{PROP:PAIRING_INV} gives $\left< \iota(\Gamma),S \right> \equiv \left<\iota_* \Gamma,S\right> = \left<\Gamma,\iota^* S\right>  = \left< \Gamma, S
\vert_M \right>.$  Any positive $(1,1)$ current on $M$ is a positive
measure.  Thus, $\int_\Gamma S \vert_M$ is defined, and coincides with the
result obtained by first choosing a smooth approximation to $S \vert_M$.
Thus $\left< \Gamma, S \vert_M \right> = \int_\Gamma S|_M$.
\end{proof}
\noindent
In the remainder of the paper, we will not typically distinguish between $\Gamma$ and $\iota(\Gamma)$.

\subsection{Linking with the Green's Current}
\label{SUBSEC:EXAMPLE_LK}
We conclude the section with some observations specific to the Green's current
$T$, including the proof of Theorem \ref{THM:GENERAL_TECHNIQUE}, as well as an
example illustrating the definitions given above.  It is worth noting that the
Green's current has empty polar locus, since $G$ is locally bounded on
$\CC^{k+1} \setminus \{0\}$, so that the hypotheses of Proposition
\ref{PROP:PAIRING_INV} and Corollary \ref{COR:RESTRICT_THEN_INTEGRATE} are easy
to check.

\begin{prop}\label{PROP:RATIONAL_LINKING}
Suppose that $f: \Pk \rightarrow \Pk$, $W^s(\zeta) \subset U(f)$ is the basin of attraction of some attracting
periodic cycle $\zeta$, and $T$ is the Green's Current of $f$.  Then 
\begin{eqnarray*}
lk(\cdot,T) :H_1(W^s(\zeta)) \rightarrow \ZZ[1/d]/\ZZ \subset \QQ/\ZZ.
\end{eqnarray*}
\end{prop}

\begin{proof}
Suppose that $\zeta$ is of period $N$.  Then, the basin of attraction
$W^s(\zeta)$ contains a union of small open balls $B_0,\ldots,B_{N-1}$ centered
at each point $\zeta,\ldots,f^{N-1}(\zeta)$ of the orbit $\zeta$.  Since
$H_1(W^s(\zeta))$ is generated by the classes of piecewise smooth loops, it is
sufficient to consider a single such loop $\gamma$.  Since $\gamma$ is a
compact subset of $W^s(\zeta)$, there is some $n$ so that $f^n(\gamma)$ is
contained in $\cup B_i$, giving that $f^n(\gamma)$ has trivial homology class
in $H_1(W^s(\zeta))$.  In particular, $lk(f^n(\gamma),T) = 0 \ (\text{mod} \
1)$, so that for any $\Gamma$ with $\partial \Gamma = \gamma$ we have
$\left<f^n(\Gamma),T \right> = k$ for some integer $k$.

Recall that $f^* T = d T$, where $d$ is the algebraic degree of $f$.
Proposition  \ref{PROP:PAIRING_INV} gives that $k = \left<f^n(\Gamma),T \right> =
\left<\Gamma,(f^*)^n T\right> = d^n \left<\Gamma,T\right>$.  In particular,
$lk(\gamma,T) \equiv k/d^n \ (\text{mod} \ 1)$.
\end{proof}

Using Proposition \ref{PROP:RATIONAL_LINKING}, Theorem \ref{THM:GENERAL_TECHNIQUE}
presents a general strategy for showing that $H_1(U(f))$ is infinitely generated.

\vspace{0.1in}
\begin{proof}[Proof of Theorem \ref{THM:GENERAL_TECHNIQUE}:]
Since $\Omega$ is a union of basins of attraction for attracting periodic points of $f$,
Proposition \ref{PROP:RATIONAL_LINKING} gives that $lk(\cdot,T): H_1(\Omega) \rightarrow \QQ/\ZZ$.
There are homology classes $c \in  H_1(\Omega)$ with $lk(c,T) \neq 0$ arbitrarily close to
zero, so, since $lk(\cdot,T)$ is a homomorphism, the image of $lk(\cdot,T): H_1(\Omega) \rightarrow \QQ/\ZZ$ is dense
in $\QQ/\ZZ$.  Because any dense subgroup of $\QQ/\ZZ$ is infinitely generated,
the image of $lk(\cdot,T)$ is infinitely generated, hence $H_1(\Omega)$ is,
as well.
\end{proof}

\begin{example}\label{EXAMPLE:DEFN_LINK}
Consider the polynomial skew product $(z,w) \mapsto (z^2,w^2+0.3z)$, for which
the Fatou set consists of the union of basins of attraction for three
super-attracting fixed points: $[0:1:0]$, $[0:0:1]$, and $[1:0:0]$.  In Figure
\ref{FIG_LINK_IN_LINE} we show a computer generated image of the intersection
of $W^s([0:1:0])$ (lighter grey) and $W^s([0:0:1])$ (dark grey) with the
vertical line $z=z_0=0.99999$.  In terms of the fiber-wise Julia sets that were
mentioned in the introduction, $K_{z_0}$ is precisely the closure of the dark
grey region and $J_{z_0}$ is its boundary.

We will see in Proposition \ref{PROP:HAMONIC_MEASURE_ON_VERTICALS} that $T
\vert_{z=z_0}$ is precisely the harmonic measure on $K_{z_0}$.  Using this
knowledge, and supposing that the computer image is accurate, we illustrate
how the above definitions can be used to show that the smooth loop $\gamma$ shown in
the figure represents a non-trivial homology class in $H_1(W^s([0:1:0]))$.  

Suppose that we use the  two chain $\Gamma_1$ that is depicted in the figure to
compute $lk(\gamma,T)$.  The harmonic measure on $K_{z_0}$ is supported in $J_{z_0}$
and equally distributed between the four symmetric pieces
with total measure of $K_{z_0}$ is $1$.  Therefore (using Corollary
\ref{COR:RESTRICT_THEN_INTEGRATE}) we see that $lk(\gamma,T) = \int_{\Gamma_1} T
\vert_{z=z_0} = \frac{1}{4} \ (\text{mod} \ 1)$, because $\Gamma_1$ covers
exactly $1$ these $4$ pieces of $K_{z_0}.$

If instead we use $\Gamma_2$, the disc ``outside of $\gamma$'' within the projective
line $z=z_0$ with the orientation chosen so that $\partial \Gamma_2 = c$ as
depicted, then $lk(\gamma,T) = \int_{\Gamma_2} T \vert_{z=z_0} = -\frac{3}{4} \
(\text{mod} \ 1)$ (because $\Gamma_2$ covers $3$ of the $4$ symmetric pieces of
$K_{z_0},$ but with the opposite orientation than that of $\Gamma_1$).
However, $-\frac{3}{4} \ (\text{mod} \ 1) = \frac{1}{4} \ (\text{mod} \ 1)$, so
we see that the computed linking number does come out the same.

Since $lk(\gamma,T) \neq 0 \ (\text{mod} \ 1)$, Corollary \ref{COR:NONZERO_LINK}
gives that it is impossible to have any $2$-chain $\Lambda$ within
$W^s([0:1:0])$ (even outside of the vertical line $z=z_0$) so that $\partial
\Lambda = c$.  Thus $[\gamma] \neq 0 \in H_1(W^s([0:1:0]))$.

\begin{figure}
\begin{picture}(0,0)%
\epsfig{file=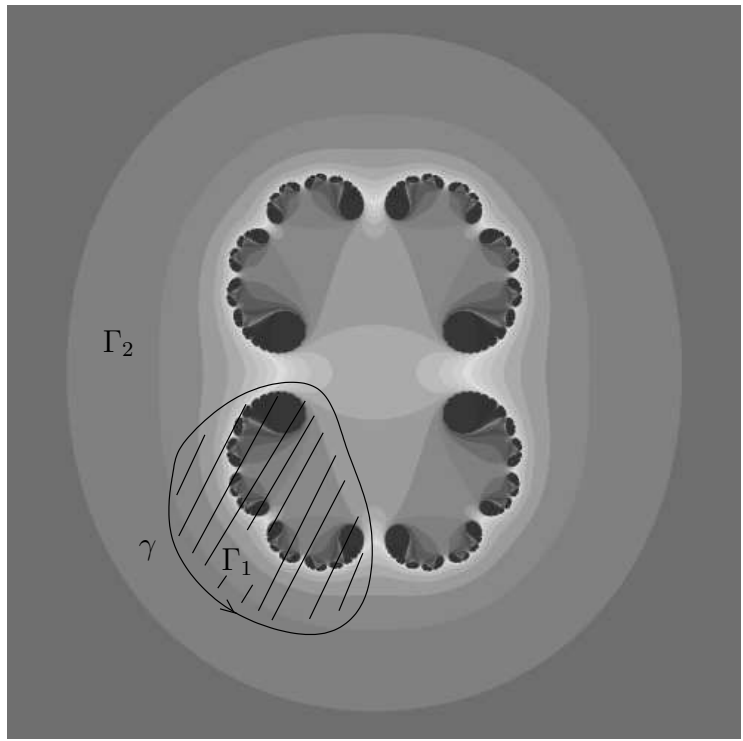}%
\end{picture}%
\setlength{\unitlength}{3947sp}%
\begingroup\makeatletter\ifx\SetFigFont\undefined%
\gdef\SetFigFont#1#2#3#4#5{%
  \reset@font\fontsize{#1}{#2pt}%
  \fontfamily{#3}\fontseries{#4}\fontshape{#5}%
  \selectfont}%
\fi\endgroup%
\begin{picture}(4607,4607)(1201,-4968)
\put(1801,-2536){\makebox(0,0)[lb]{\smash{{\SetFigFont{12}{14.4}{\familydefault}{\mddefault}{\updefault}{\color[rgb]{0,0,0}$\Gamma_2$}%
}}}}
\put(2026,-3811){\makebox(0,0)[lb]{\smash{{\SetFigFont{12}{14.4}{\familydefault}{\mddefault}{\updefault}{\color[rgb]{0,0,0}$\gamma$}%
}}}}
\put(2546,-3912){\makebox(0,0)[lb]{\smash{{\SetFigFont{12}{14.4}{\familydefault}{\mddefault}{\updefault}{\color[rgb]{0,0,0}$\Gamma_1$}%
}}}}
\end{picture}%
\caption{\label{FIG_LINK_IN_LINE} Both choices $\Gamma_1$ (inside of $\gamma$) and $\Gamma_2$ (outside of $\gamma$) yield the same $lk(\gamma,T)$.}
\end{figure}

\end{example}

\section{Application to Polynomial Endomorphisms of $\Pt$}
\label{SEC:ENDO}

Having developed the linking numbers in Section \ref{SEC:LINKING}, Theorem \ref{THM:JPI_DISCONN} will be a  consequence of the following
well-known result:
\begin{thm}\label{THM:J_DISCONN_RS}\cite[Thm. 5.7.1]{BEAR}
Let $g:\Po \rightarrow \Po$ be a rational map.  Then, if
$J(g)$ is disconnected, it contains uncountably many components, and each point
of $J(g)$ is an accumulation point of infinitely many distinct components of
$J(g)$.
\end{thm}

Let us begin by studying the Fatou set of one-dimensional maps:
\begin{prop}\label{PROP:1VAR_FATOU}
If $g:\Po \rightarrow \Po$ is a hyperbolic rational map with disconnected Julia set $J(g)$, then the Fatou set $U(g)$ has infinitely generated first homology.
\end{prop}

\begin{rmk}\label{EG:ONE_VARIABLE_EXAMPLES}
When reading the proof of
Proposition \ref{PROP:1VAR_FATOU}, it is helpful to keep in mind two examples.  
The first is the polynomial $r(z) = z^3-0.48z+(0.706260+0.502896i)$ for
which one of the critical points escapes to infinity, while the other is in the
basin of attraction for a cycle of period $3$.  The result is a filled Julia
set with infinitely many non-trivial connected components, each of which is
homeomorphic to the Douady's rabbit.  (See \cite{MILNOR}.)

The second example are maps of the form $f(z) = z^n + \lambda/z^h$, which were
considered in \cite{MCMULLEN}.  For suitable $n, h$, and $\lambda$ the Julia set
is a Cantor set of nested simple closed curves.
\end{rmk}

\begin{proof}[Proof of Proposition \ref{PROP:1VAR_FATOU}:]
Since $g$ is hyperbolic, $U(g)$ consists of the basins of attraction of
finitely many attracting periodic points.  Therefore, according to Theorem
\ref{THM:GENERAL_TECHNIQUE}, it is sufficient to find elements of $H_1(U(g))$
having non-zero linking numbers with $T=\mu_g$ that are arbitrarily close to $0$.

Theorem \ref{THM:J_DISCONN_RS} will allow us to find a sequence of piecewise
smooth two chains $\Gamma_1, \Gamma_2,\ldots$ so that $0 <
\left<\Gamma_{n-1},\mu_G\right> < \left<\Gamma_{n},\mu_G\right> < 1$ and
$\partial \Gamma_n \subset U(g)$, as follows.  

For each $n$, $\Gamma_n$ will be a union of disjoint positively-oriented closed
discs in $\Po$, each counted with weight one.  Since $J(g)$ is disconnected, we
can find a piecewise smooth oriented loop $\gamma_1 \subset U(g)$ that
separates $J(g)$.  Let $\Gamma_1$ be the positively-oriented disc in $\Po$
having $\gamma_1$ as its oriented boundary.  Since $\mu_g$ is normalized and
$\gamma_1$ separates $J(g) = \supp(\mu_g)$, we have $0 <
\left<\Gamma_1,\mu_g\right> < 1$.  Now suppose that
$\Gamma_1,\ldots,\Gamma_{n-1}$ have been chosen.  Since
$\left<\Gamma_{n-1},\mu_g\right> < 1$, we have $J(g) \cap (\Po \setminus
\Gamma_{n-1}) \neq \emptyset$.  Then, according to Theorem
\ref{THM:J_DISCONN_RS}, there is more than one component of $J(g) \cap (\Po
\setminus \Gamma_{n-1})$, so we can choose an oriented loop $\gamma_n \subset
U(g) \cap (\Po \setminus \Gamma_{n-1})$ so that at least one component of $J(g)
\cap (\Po \setminus \Gamma_{n-1})$ is on each side of $\gamma_n$.  Then, we let
$\Gamma_n$ be the union of oriented discs in $\Po$ consisting of the points
inside of $\gamma_n$ and any discs from $\Gamma_{n-1}$ that are not inside of
$\gamma_n$.

Considering the homology class $[\partial \Gamma_n - \partial \Gamma_{n-1}] \in H_1(U(g))$ we have that
\begin{eqnarray*}
lk([\partial \Gamma_n - \partial \Gamma_{n-1}],\mu_g) = \left<\Gamma_{n},\mu_G\right> - \left<\Gamma_{n-1},\mu_G\right>  \ (\text{mod} \ 1)
\end{eqnarray*}
\noindent
is non-zero for each $n$.  However, since 
\begin{eqnarray*}
\sum_n \left<\Gamma_{n},\mu_G\right> - \left<\Gamma_{n-1},\mu_G\right>
\end{eqnarray*}
\noindent
is bounded by $1$, we have that $lk([\partial \Gamma_n - \partial
\Gamma_{n-1}],\mu_g) \rightarrow 0$ in $\QQ/\ZZ$.  Theorem
\ref{THM:GENERAL_TECHNIQUE} then gives that $H_1(U(g))$ is infinitely
generated. 
\end{proof}

Let $f:\Pt \rightarrow \Pt$ be a polynomial endomorphism given 
in projective coordinates by
\begin{eqnarray}\label{EQN:LIFT_F}
f([Z:W:T]) = [P(Z,W,T):Q(Z,W,T):T^d].
\end{eqnarray}
\noindent
Since $f:\Pt \rightarrow \Pt$ is assumed globally holomorphic, $P(Z,W,T), Q(Z,W,T),$ and $T^d$ have no common zeros other than $(0,0,0)$.

The (projective) line at infinity $\Pi:=\{T=0\}$ is uniformly
super-attracting and the restriction $f_\Pi$
is given in homogeneous coordinates by 
\begin{eqnarray}\label{EQN:LIFT_F_PI}
f_\Pi:([Z:W]) \rightarrow [P_0(Z,W):Q_0(Z,W)].
\end{eqnarray}
\noindent
where $P_0 := P(Z,W,0)$ and $Q_0 :=Q(Z,W,0)$.

Let $U(f)$ be the Fatou set for $f$ and $U({f_\Pi})$ the Fatou set for $f_\Pi$.  The former is an open set in $\Pt$, while
the latter is an open set in the line at infinity $\Pi$.

\begin{lem}\label{LEM:FAT_PI}If $f_\Pi$ is hyperbolic then
$U({f_\Pi}) \subset U(f)$.
\end{lem}

\begin{proof}
Since $f_\Pi$ is hyperbolic, $U({f_\Pi})$ is in the union of the basins of
attraction $W^s_\Pi(\zeta_i)$ of a finite number of periodic attracting points
$\zeta_1,\ldots,\zeta_k$.  The line at infinity $\Pi$ is transversally
superattracting, so each $\zeta_i$ is superattracting in the
transverse direction to $\Pi$ and (at least) geometrically attracting within
$\Pi$.  Let $W^s(\zeta_i) \subset \Pt$ be the basin of attraction for $\zeta_i$ under $f$.  Then, $W^s_\Pi(\zeta_i) \subset W^s(\zeta_i)$,
giving $U({f_\Pi}) \subset U(f)$.
\end{proof}

Let $T$ be the Green's current for $f$ and let $\mu_\Pi$ be the measure of maximal entropy for the restriction $f_{|\Pi}$.

\begin{lem}\label{LEM:PI_RESTRICTION}
The restriction $T_{| \Pi}$ coincides with $\mu_\Pi$.
\end{lem}

\begin{proof}
Consider the lift $F_\Pi:\CC^2 \rightarrow \CC^2$ of the rational map $f_\Pi :\Po \rightarrow \Po$.  As observed in Remark \ref{RMK:GREEN_1D}, 
\begin{eqnarray*}
G_\Pi(Z,W) = \lim \frac{1}{d^n} \log ||F_\Pi^n(Z,W)||
\end{eqnarray*}
is the potential for $\mu_\Pi$, meaning that $\pi^* \mu_\Pi = \frac{1}{2\pi} dd^c G_\Pi$.

The restriction $T_{| \Pi}$ is obtained by restricting of the potential $G$ to $\pi^{-1}(\Pi) = \{(Z,W,0) \in \CC^3\}$.  Specifically,
it is defined by $\pi^* (T_{|\Pi}) = \frac{1}{2\pi} dd^c(G(Z,W,0))$.  Therefore, it suffices to show that $G(Z,W,0) = G_\Pi(Z,W)$.
However, this follows directly from the fact that $F(Z,W,0) = F_\Pi(Z,W)$.  (Here $F$ is the lift of
$f$ to $\CC^3$, as given by (\ref{EQN:LIFT_F}) when considered in non-projective coordinates $[Z,W,T]$.)
\end{proof}

\vspace{0.1in}
\begin{proof}[Proof of Theorem \ref{THM:JPI_DISCONN}]
As in the proof of Proposition \ref{PROP:1VAR_FATOU}, we can find a sequence of
$1$-cycles $c_n$ in $U(f_\Pi)$ having linking numbers with $\mu_\Pi$
arbitrarily close to $0$ in $\QQ/\ZZ$.  Since $f_{|\Pi}$ is hyperbolic, Lemma
\ref{LEM:FAT_PI} gives that each $c_n$ is in the union of basins of attraction
for finitely many attracting periodic points of $f$.  In particular,
$lk(c_i,T)$ is well-defined for each $n$.  Lemma
\ref{LEM:PI_RESTRICTION} gives that $T_{|\Pi} = \mu_\Pi$, so that 
$lk(c_n,T)$ (considering $c_n$ in $\Pt$) coincides with $lk(c_n,\mu_\Pi)$
(considering $c_n$ in the projective line $\Pi$).  Therefore, $lk(c_n,T) \neq 0$
and $lk(c_n,T) \rightarrow 0$ in $\QQ/\ZZ$.  Theorem
\ref{THM:GENERAL_TECHNIQUE} gives that the union of these basins has infinitely
generated first homology, and hence $U(f)$ does as well.  
\end{proof}

\vspace{0.1in}

\begin{example}\label{EXAMPLE:RABBITS}
We embed the polynomial dynamics of $r(z)$  from Remark
\ref{EG:ONE_VARIABLE_EXAMPLES} as the dynamics on the line at infinity $\Pi$
for a polynomial endomorphism of $\Pt$.  Let $R(Z,W) =
Z^3-0.48ZW^2+(0.706260+0.502896i)W^3$ be the homogeneous form of $r$, and let
$P(Z,W,T)$ and $Q(Z,W,T)$ be any homogeneous polynomials of degree $2$.  Then
\begin{eqnarray*}
f([Z:W:T]) = [R(Z,W)+T\cdot P(Z,W,T):W^3+T \cdot Q(Z,W,T):T^3]
\end{eqnarray*}

\noindent
is a polynomial endomorphism with $f_\Pi = r$.  In this case, Theorem \ref{THM:JPI_DISCONN} gives that the basin of attraction of $[1:0:0]$ for $f$ has infinitely generated first homology.
\end{example}

\begin{rmk}
Suppose that $f: \Pk \rightarrow \Pk$ is a holomorphic endomorphism having an
invariant projective line $\Pi$.  Lemma \ref{LEM:PI_RESTRICTION} can be
extended to give that $T_{|\Pi} = \mu_\Pi$, where $\mu_\Pi$ is the measure of
maximal entropy for the one-dimensional map $f_{|\Pi}$.  If $\Pi$ is at least
geometrically attracting transversally, $f_{|\Pi}$ is hyperbolic, and
$J(f_{|\Pi})$ is disconnected, then essentially the same proof as that of
Theorem \ref{THM:JPI_DISCONN} gives that the Fatou set $U(f)$ has infinitely
generated first homology.

Using this observation, one can inductively create sequences of polynomial
endomorphisms $f_k: \Pk \rightarrow \Pk$, for every $k$, each having Fatou
set with infinitely generated first homology.  One begins with a hyperbolic
polynomial endomorphism $f_1$ of the Riemann sphere $\mathbb{P}^1$ having
disconnected Julia set.   Then, for each $k$, one can let $f_{k}:\Pk
\rightarrow \Pk$ be any polynomial endomorphism whose dynamics on the
hypersurface $\mathbb{P}^{k-1}$ at infinity is given by $f_{k-1}$. (When $k=2$,
the construction of $f_2:\Pt \rightarrow \Pt$ is similar to that from Example
\ref{EXAMPLE:RABBITS}.) The resulting maps each have a totally-invariant
projective line $\Pi$ that is transversally superattracting with ${f_k}_{|\Pi}
= f_1$ hyperbolic with disconnected Julia set.  Thus, the Fatou set $U(f_k)$
has infinitely generated first homology.  
\end{rmk}

\section{Application to Polynomial skew products} 
\label{SEC:SKEW}

A {\em polynomial skew product} is a polynomial endomorphism of the form 
\begin{eqnarray*}
f(z,w) = (p(z),q(z,w))
\end{eqnarray*}
\noindent with $p$ and $q$ polynomials of degree $d$ where $p(z) = z^d +
O(z^{d-1})$ and $q(z) = w^d +O_z(w^{d-1})$.  (See Jonsson \cite{JON_SKEW}.)
Theorem \ref{THM:JPI_DISCONN} can by applied to many polynomial skew products
$f$ to show that that $H_1(U(f))$ is infinitely generated; for example, $f(z,w)
= (z^2,w^2+10z^2)$, which has $J_{\Pi}$ a Cantor set.  Next we will find
alternative sufficient conditions under which a polynomial skew product has
Fatou set with infinitely generated first homology, proving Theorem
\ref{THM:MAIN}.  This theorem will apply to many maps to which Theorem
\ref{THM:JPI_DISCONN} does not apply; for example, $f(z,w) = (z^2,w^2-3z)$, for
which $J_{\Pi}$ is equal to the unit circle.

\vspace{0.1in}

\subsection{Preliminary background on polynomial skew products}
\label{SUBSEC:BACKGROUND_SKEW}


The Green's
current for any polynomial endomorphism can be computed in the affine coordinates on $\CC^2$ as $T :=
\frac{1}{2\pi} dd^{c} G_{\rm affine}$, where $G_{\rm affine}$ is the (affine)
Green's function defined in Remark \ref{RMK:AFFINE_GREENS_FUNCTION}.  The
``base map'' $p(z)$ has a Julia set $J_p \subset \mathbb{C}$ and, similarly, a
Green's function $G_p(z):= \lim_{n \to \infty} \frac{1}{d^n} \log_+
||p^n(z)||$.  Furthermore, one can define a fiber-wise Green's
function\footnote{For the purist: the Green's functions $G_p$ and $G_z$ should
also have the subscript ``affine'', but it is dropped here for ease of
notation.  See Section \ref{SEC:GREENS_CURRENT} for the distinction.} by:
\begin{eqnarray*}
G_z(w) := G_{\rm affine}(z,w) - G_p(z).
\end{eqnarray*}
\noindent 
For each fixed $z$, $G_z(w)$ is a subharmonic function of $w$  and one defines
the fiber-wise Julia sets by $K_z := \{G_z(w) = 0\}$ and $J_z := \partial K_z$.   

The extension of $f$ to $\Pt$ is given by
\begin{eqnarray} \label{EQN:SKEW_HOMOG}
f([Z:W:T]) = [P(Z,T):Q(Z,W,T):T^d],
\end{eqnarray} \noindent
where $P(Z,T)$ and $Q(Z,W,T)$ are the homogeneous versions of $p$ and $q$.
The point\\ $[0:1:0]$ that is ``vertically at infinity'' with respect to the affine coordinates $(z,w)$ is a totally-invariant super-attracting fixed point and $(z,w) \in W^s([0:1:0])$ if and only if
$w \in \CC \setminus K_z$.

Suppose that $(z,w) \in W^s([0:1:0])$ and $(z_n,w_n) := f^n(z,w)$.  Then, 
\begin{eqnarray}
G_{\rm affine}(z,w) &=& \lim \frac{1}{d^n} \log_+ \|f^n(z,w)\|_\infty = \lim \frac{1}{d^n} \log_+ |w_n| \,\, \mbox{and} \\
G_z(w) &=& G_{\rm affine}(z,w) - G_p(z) =  \lim \frac{1}{d^n} \log_+ |w_n| - \lim \frac{1}{d^n} \log_+|z_n|. \label{EQN:VERTICAL_GREEN}
\end{eqnarray}
\noindent
since $|w_n| > |z_n|$ for all $n$ sufficiently large.


As mentioned in Section \ref{SUBSECTION:INVARIANCE_AND_RESTRICTION}, we can
restrict the current $T$ to any analytic curve obtaining a measure on that
curve.  Of particular interest for skew products is the restriction $\mu_{z_0}$
of $T$ to a vertical line $\{z_0\} \times \mathbb{P}$.  The following appears
as Jonsson \cite{JON_SKEW} Proposition 2.1 (i), we repeat it here for
completeness:

\begin{prop} \label{PROP:HAMONIC_MEASURE_ON_VERTICALS}
The restriction $T _{| {z=z_0}}$ of the Green's current $T$ to a vertical line
$(\{z_0\} \times \mathbb{P})$ coincides with the harmonic measure $\mu_{z_0}$ on $K_{z_0}$.
\end{prop}

\begin{proof}
Notice that
\begin{eqnarray*}  T_{| {z=z_0}} &=& \frac{1}{2\pi} dd^c G_{{\rm affine}|{z=z_0}} = 
 \frac{1}{2\pi} dd^c G_{\rm affine}(z_0,w)\\ &=& \frac{1}{2\pi} dd^c \left(G_{\rm affine}(z_0,w) -
G_p(z_0)\right) = \frac{1}{2\pi} dd^c G_{z_0}(w). \end{eqnarray*}

\noindent
According to \cite[Thm 2.1]{JON_SKEW}, $G_{z_0}$ is the Green's function for
$K_z$ with pole at infinity.  We have thus obtained that $\mu_{z_0}$ is exactly
the harmonic measure $\mu_{z_0}$ on $K_{z_0}$.
\end{proof}

\subsection{Topology of the basin of attraction $W^s([0:1:0])$}
\label{SEC:MAIN_RESULT}

\begin{prop}\label{W_PATH_CONNECTED} If $\zeta$ is a totally-invariant (super)-attracting
fixed point for a holomorphic $f:\PK \rightarrow \PK$, then $W^s(\zeta)$ is path-connected.
\end{prop}

\noindent
A nearly identical statement is proven for $\PP$ in Theorem 1.5.9 from
\cite{HP}.  We refer the reader to their proof since it is nearly identical for
$\PK$.  In particular, for any skew product 
$W^s([0:1:0])$ is path connected.

\vspace{0.1in}
Although $G_z(w)$ is subharmonic in $w$ for any fixed $z$, it does not form a
PSH function of both $z$ and $w$.  Consider the points $(z,w) \in W^s([0:1:0])$
for which $z \in J_p$.  At these points $G_{\rm affine}$ is pluriharmonic, i.e. $dd^c
G_{\rm affine} = 0$, but $G_p(z)$ is not pluriharmonic, i.e. $dd^c G_p(z) > 0$.
Therefore, at these points $dd^c G_z(w) < 0$, so $G_z(w)$ is not PSH.

\begin{lem}\label{LEM:MINUS_G_Z_PSH}
The function $-G_z(w)$ is PSH at all points $(z,w) \in W^s([0:1:0]) \cap \CC^2$ and it
extends to a PSH function on all of $W^s([0:1:0])$.  The resulting function is
pluriharmonic on $W^s([0:1:0])$ except at points for which $Z/T \in J_p$.
\end{lem}

\begin{proof}
Since $-G_z(w) = G_p(z) - G_{\rm affine}(z,w)$, with $G_{\rm affine}(z,w)$
pluriharmonic in $W^s([0:1:0])$ and $G_p(z)$ PSH everywhere, the result is PSH
in $W^s([0:1:0]) \cap \CC^2$.

Jonsson proves in \cite[Lemma 6.3]{JON_SKEW} that $G_z(w)$ extends as a PSH
function in a suitable neighborhood of $\Pi \setminus \{[0:1:0]\}$ and his proof
immediately gives that the result is pluriharmonic in a (possibly smaller)
neighborhood within $W^s([0:1:0])$ of $\Pi \setminus \{[0:1:0]\}$.  Therefore, $-G_z(w)$
is also pluriharmonic in the same neighborhood.

Thus, $-G_z(w)$ extends to a PSH on $W^s([0:1:0]) \setminus \{[0:1:0]\}$ and, assigning
$-\infty$ to $[0:1:0]$, gives the desired extension to all of $W^s([0:1:0])$.
The result will be pluriharmonic except at $[0:1:0]$ and at the points in $W^s([0:1:0]) \cap \CC^2$
where $dd^c(-G_z(w)) > 0$, that is the points where $Z/T \in J_p$.
\end{proof}

\vspace{0.1in}
\begin{proof}[Proof of Theorem \ref{THM:MAIN}:]
We first suppose that $J_{z_0}$ is disconnected for some $z_0 \in J_p$.
Let $z_1,z_2,\ldots$ be any sequence of iterated preimages of $z_0$ so that $p^n(z_n) = z_0$.

Consider the vertical line $\{z_0\} \times \CC$.  Since $J_{z_0}$ is
disconnected, so is $K_{z_0}$, and we can choose two disjoint
positively-oriented piecewise smooth loops $\eta_1, \eta_2 \subset \{z_0\}
\times \left(\CC \setminus K_{z_0}\right)$ each enclosing a proper subset of
$K_{z_0}$.

Perturbing $\eta_1, \eta_2$ within $\{z_0\} \times (\CC \setminus K_{z_0})$, if
necessary, we can suppose that none of the $d-1$ critical values of
$f|_{\{z_1\}\times \CC}: \{z_1\} \times \CC \rightarrow \{z_0\} \times \CC$
(counted with multiplicity) are on $\eta_1$ or $\eta_2$.  Since the regions
enclosed by $\eta_1$ and $\eta_2$ are disjoint, at least one of them contains
at most $d-2$ of these critical values.  Let $\gamma_0$ be this curve.

Since $\gamma_0 \subset \{z_0\} \times (\CC \setminus K_{z_0})$, $\gamma_0
\subset W^s([0:1:0])$.  Because $\gamma_0$ is compact, it is bounded away from
$\supp(T)$, and the linking number $lk(\gamma_0,T)$ is a well defined invariant
of the homology class $[\gamma]$ within $H_1(W^s([0:1:0]))$. 
We let $\Gamma_0$ be the closed disc in $\left(\{z_0\} \times \mathbb{C}\right)$ having $\gamma_0$ as its oriented boundary.  Since $\Gamma_0$
contains some proper subset of $K_{z_0}$ (and hence of $J_{z_0}$) with $\supp(\mu_{z_0}) = J_{z_0}$, we have that
\begin{eqnarray*}
0 < \left<\Gamma_0,T\right> = \int_{\Gamma_0} \mu_{z_0} < 1.
\end{eqnarray*}
\noindent
Therefore, $lk(\gamma_0,T) = \left<\Gamma_0,T\right>(\text{mod} \ 1) \neq 0 \
(\text{mod} \ 1)$, giving that $[\gamma_0]$ is non-trivial.

\vspace{0.1in}

Consider the preimages $D_1,\ldots,D_j$ of $\Gamma_0$ under $f|_{\{z_1\} \times
\CC} : \{z_1\} \times \CC \rightarrow \{z_0\} \times \CC$.  Since at most $d-2$
critical values of the degree $d$ ramified cover $f|_{\{z_1\} \times \cup D_i}$ are
contained in $\Gamma_0$, it is a consequence of the Riemann-Hurwitz Theorem
that the Euler characteristic of $\cup D_i$ is greater than or equal to $2$.
Because each $D_i$ is a domain in $\CC$, at least two
components $D_1$ and $D_2$ are discs.  

The total degree of $f|_{\{z_1\}\times \CC}: \cup D_i \rightarrow \Gamma_0$ is
$d$, so $f|_{\{z_1\} \times \CC}: D_i \rightarrow \Gamma_0$ a ramified covering
of degree $k_i \leq d-1$ for each $i$.  Proposition \ref{PROP:PAIRING_INV}
and the basic invariance $f^*T = d\cdot T$ for the Green's current give that
\begin{eqnarray}\label{EQN:SMALLER_PAIRING}
\left<D_i,T\right> = \frac{1}{d}\left<D_i,f^*T\right> = \frac{1}{d} \left<f_* D_i,T\right>  = \frac{1}{d} \left<k_i \Gamma_0,T\right> \leq \frac{d-1}{d}\left<\Gamma_0,T\right>
\end{eqnarray}
for each $i$.

As before, we can perturb the boundaries of $D_1$ and $D_2$ within $\{z_1\}
\times (\CC \setminus K_{z_1})$ so that none of the critical values of
$f|_{\{z_2\} \times \CC}$ lie on their boundaries and so that $D_1$ and $D_2$
remain disjoint.  (It will not affect the pairings given by
(\ref{EQN:SMALLER_PAIRING})).  At least one of the discs $D_1, D_2$  contains
at most $d-2$ critical values of $f|_{\{z_2\} \times C}$.  We let
$\Gamma_1$ be that disc and $\gamma_1 = \partial \Gamma_1$.  Then
\begin{eqnarray*} 
0 < \left<\Gamma_1,T\right> \leq \frac{d-1}{d} \left<\Gamma_0,T\right> \leq \frac{d-1}{d}.  
\end{eqnarray*}

Continuing in the same way, we can find a sequence of discs
$\Gamma_0,\Gamma_1,\ldots$ so that
\begin{itemize}
\item $\Gamma_n \subset \{z_n\} \times \CC$,
\item $\gamma_n = \partial \Gamma_n \subset W^s([0:1:0])$, 
\item $\Gamma_n$ contains at most $d-2$ critical values of $f|_{\{z_{n+1\}}\times \CC}$ (counted with multiplicity), and
\item $\left<\Gamma_n,T\right> \leq \frac{d-1}{d} \left<\Gamma_{n-1},T\right>$.
\end{itemize}
\noindent
Consequently,
\begin{eqnarray*}
0 < \left<\Gamma_n,T\right> \leq \left(\frac{d-1}{d}\right)^n,
\end{eqnarray*}
\noindent
giving that $lk(\gamma_n,T) \rightarrow 0$ in $\QQ/\ZZ$.  Therefore, Theorem
\ref{THM:GENERAL_TECHNIQUE} gives that $H_1(W^s([0:1:0]))$ is infinitely
generated.

\vspace{0.1in}

We will now show that if $J_z$ is connected for every $z \in J_p$, then
$W^s([0:1:0])$ is homeomorphic to an open ball.
Consider the local coordinates $z' = Z/W$, $t' = T/W$, chosen so that $(z',t')
= (0,0)$ corresponds to $[0:1:0]$.  In these coordinates
\begin{eqnarray*}
f(z',t') = \left(\frac{P(z',t')}{Q(z',1,t')},\frac{t'^d}{Q(z',1,t')}\right),
\end{eqnarray*}
\noindent
where $P$ and $Q$ are the homogeneous versions of $p$ and $q$ appearing in Equation (\ref{EQN:SKEW_HOMOG}).
The assumption that $q(z) = w^d + O_z(w^{d-1})$ and $p(z) = z^d+O(z^{d-1})$ imply that 
we have the expansion
\begin{eqnarray*}
f(z',t') = (P(z',t'),t'^d) + g(t',z'),
\end{eqnarray*}
\noindent
with $(P(z',t'),t'^d)$ non-degenerate of degree $d$ and
$g(t',z')$ consisting of terms of degree $d+1$ and larger.

Therefore, we can 
construct a
potential function\footnote{The potential function $h$ is sometimes also be called the Green's function of the point $(0,0)$.} for the superattracting point $(0,0)$:
\begin{align}\label{EQN:GREEN_FOR_POINT}
h(z',t') := \lim_{n \rightarrow \infty} \frac{1}{d^n} \log \|f^n(z',t')\|_\infty.
\end{align}
\noindent
The result is a continuous pluri-subharmonic function \cite{HP} with logarithmic
singularity at $(z',t') = (0,0)$ having the property that $(z',t') \in
W^s([0:1:0])\setminus \{[0:1:0]\}$ if and only if $h(z',t') < 0$.   In
particular,
\begin{eqnarray*}
h: W^s([0:1:0]) \setminus \{[0:1:0]\} \longrightarrow (-\infty,0)
\end{eqnarray*}
\noindent
is proper.

If we let $(z'_n,t'_n) = f^n(z',t')$, then Equation (\ref{EQN:GREEN_FOR_POINT}) simplifies to
\begin{align}
h(z',t') =
\begin{cases}
\lim \frac{1}{d^n} \log |t'_n| & \text{if $z'/t' \in K_p$ and}\\
\lim \frac{1}{d^n} \log |z'_n| & \text{if $z'/t' \not \in K_p$.}
\end{cases}
\end{align}
\noindent
since $z'_{n+1}/t'_{n+1} = p(z'_n/t'_n)$.
Equation (\ref{EQN:VERTICAL_GREEN}) gives that in the original affine coordinates $(z,w)$ we have
\begin{align}\label{EQN:EQUALITY_FOR_POTENTIAL}
h(z,w) =
\begin{cases}
\lim \frac{1}{d^n} \log |t'_n| = - \lim  \frac{1}{d^n} \log |w_n| =  -G_z(w) & \text{if $z \in K_p$ and},\\
\lim \frac{1}{d^n} \log |z'_n| = \log|z| - \lim  \frac{1}{d^n} \log |w_n| = - G_z(w) & \text{if $z \not \in K_p$},
\end{cases}
\end{align}
which is harmonic on the intersection of any vertical line $\{z\} \times \CC$
with $W^s([0:1:0])$ and pluriharmonic except when $z \in J_p$; see Lemma \ref{LEM:MINUS_G_Z_PSH}.  A similar
calculation shows
that $h$ coincides with the extension of $-G_z(w)$ described in Lemma \ref{LEM:MINUS_G_Z_PSH}
and that the restriction of $h$ to $\Pi$ is $-G_\Pi$.  (Here, $G_\Pi$ is the Green's function
for the action $f_\Pi$ of $f$ on the line at infinity.)

Therefore, $h(z',t')$ is pluriharmonic on $W^s([0:1:0]) \setminus \{(z',w') \,:\,z'/w' \in
J_p\}$ and the restriction of $h(z',t')$ to any line through $(0,0)$ is
harmonic on $W^s([0:1:0]) \setminus \{[0:1:0]\}$, with a logarithmic singularity at $(0,0)$.

Since $J_{z_0}$ is connected for every $z_0 \in J_p$,  Proposition 6.3 from
\cite{JON_SKEW} gives that $J_z$ is connected for every $z \in \CC$ and also
$J_\Pi$ is connected, or, equivalently, that $G_z(w)$ (for any $z$) and
$G_\Pi$ have no (escaping) critical points.  Therefore, the
restriction of $h$ to any complex line through $(0,0)$ has no critical points in
$W^s([0:1:0])$.

The sublevel set $W_a := h^{-1}([-\infty,a))$ is open for any $a \in
(-\infty,0)$
since $h: W^s([0:1:0]) \setminus \{[0:1:0]\} \rightarrow
(-\infty,0)$ is continuous with $h(z',t') \rightarrow -\infty$ if and only if $(z',t') \rightarrow (0,0)$.
Equation 2.2 from \cite{JON_SKEW} implies that
\begin{eqnarray*}
h(z',t') &=& \log|t'| + G_p\left(\frac{z'}{t'}\right) + \eta(z',t') \,\, \mbox{if $t' \neq 0$, and}\\
h(z',t') &=& \log|z'| + G^\#_p\left(\frac{t'}{z'}\right) + \eta(z',t') \,\, \mbox{if $z' \neq 0$,}
\end{eqnarray*}
\noindent
with $\eta(z',t')$ becoming arbitrarily small for $(z',t')$ sufficiently small
and $G^\#_p(x)$ obtained by extending $G_p(1/x) - \log(1/x)$ continuously
through $x=0$.  Therefore, for $a$ sufficiently negative, the intersection of
$W_a$ with any complex line through $(0,0)$ will be convex.  In particular, $W_a$ is an
star-convex open subset in $\CC^2$, implying that it is homeomorphic to an open
ball.  (See \cite[Theorem 11.3.6.1]{BERGER}.)

We define a new function
$\widetilde{h}$ which agrees with $h$ except in the interior of $W_a$, where we
make a $C^\infty$ modification (assigning values less than $a$) in order to remove
the logarithmic singularity at $[0:1:0]$.

We will use $\widetilde{h}$ as Morse function to show that $W_b :=
h^{-1}([-\infty,b))$ is diffeomorphic to $W_a$ for any $b \in (a,0)$.
The classical technique from Theorem 3.1 of \cite{MILNOR_MORSE} would use the
normalization of $-\nabla \widetilde{h}$ to generate a flow whose time $(b-a)$ map
gives the desired diffeomorphism.  This will not work in our situation, since
$\widetilde{h}$ is not differentiable at points for which $z'/w' \in J_p$.
However, essentially the same proof works if we replace $-\nabla
\widetilde h$ with any $C^1$ vector field $V$ on $W^s([0:1:0])$ having no
singularities in $\widetilde{h}^{-1}([a,b])$ and along which
$\widetilde h$ is decreasing.   Note that, as in \cite{MILNOR_MORSE},
we need that $\widetilde{h}^{-1}([a,b])$ is compact, which follows
from $h$ being proper.

Let $V$ be the the vector field parallel to each line through $(z',t') =
(0,0)$, obtained within each line as minus the gradient of the restriction of
$\widetilde{h}$ to that line.  The restriction of $\widetilde{h}$ to each
complex line through $(0,0)$ has no critical points in
$\widetilde{h}^{-1}([a,b])$, so it is decreasing along $V$.  Since $h$ is pluriharmonic for points with $z'/t' \not \in J_p$, it follows
immediately that $V$ is smooth there.
To see that $V$ is smooth in a
neighborhood of points where $z'/t' \in J_p$, notice that 
\begin{eqnarray*}
\nabla_w G_z(w) = \nabla_w G(z,w) - G_p(z) = \nabla_w G(z,w), 
\end{eqnarray*}
with $G(z,w)$ pluriharmonic on $W^s([0:1:0]) \cap \CC^2$.

Therefore, for any $b \in (a,0)$, $W_b$ is homeomorphic to $W_a$ and thus to an
open ball.  One can then make a relatively standard construction, using these
homeomorphisms for $b$ increasing to $0$, in order to show that $W^s([0:1:0]) =
\cup_{b < 0} W_b$ is homeomorphic to an open ball.
\end{proof}

\section{Further applications}
\label{SEC:FURTHER_APPS}

In this final section we discuss a few examples of maps to which we have applied the results of this paper, and then a few types of maps which we feel would be fruitful to study further with techniques similar to those of this paper. 

\subsection{Relationship between connectivity of $J_2$ and the topology of the Fatou set for polynomial skew products
}
\label{SEC:RELATIONSHIP_TO_MU}

For polynomial skew products,  $J_2 = \supp(\mu) = \supp(T \wedge T) = \overline{\bigcup_{z \in J_p} J_z}$, which by
\cite{JON_SKEW} is also the closure of the set of repelling periodic points. Here we examine to what extent connectivity of $J_2$ affects the homology of the Fatou set $U$.

The following example shows that there are many polynomial skew products $f$ with $J_2$ connected for which
$H_1(U(f))$ is non-trivial (in fact infinitely generated.) 

\begin{example} \label{EXAMPLE:J1DISCONNJ2CONN}
Consider $f(z,w) = (z^2-2,w^2+2(2-z))$  which has $J_2$ connected and
has $J_z$ disconnected over $z=-2 \in J_p$, as shown in 
\cite[Example 9.7]{JON_SKEW}.  Theorem \ref{THM:MAIN} immediately applies, giving that $H_1(U(f))$ is
infinitely generated.

\vspace{0.05in}
In fact, examples of this phenomenon can appear ``stably'' within a one parameter family.
Let $p_n(z) = z^2 + c_n$ be the unique quadratic polynomial with
periodic critical point of least period $n$ and $c_n$ real.  Then, 
\cite[Theorem 6.1]{S} yields that for $n$ sufficiently large,
\begin{eqnarray*}
f_n(z,w) = (p_n(z),w^2+2(2-z))
\end{eqnarray*}
is Axiom A with $J_z$ disconnected for most $z\in J_{p_n}$ and with $J_2$
connected.  Suppose that $f_n$ is embedded within any holomorphic one-parameter
family $f_{n,\lambda}$ of polynomial skew products.  Then, Theorems 4.1 and 4.2
from \cite{S} (see also, \cite[Thm C]{JON_MOTION}) give that all maps $f_{n,\lambda}$ within the same hyperbolic
component as $f_n$ also have $J_2$ connected, but $J_z$ disconnected over most
$z$ in $J_{p_{n,\lambda}}$. (Here, $p_{n,\lambda}$ is the first component of
$f_{n,\lambda}$.)  An immediate application of Theorem \ref{THM:MAIN} yields
that $H_1(U(f_{n,\lambda}))$ is infinitely generated for all $f_{n,\lambda}$
within this hyperbolic component.
 \end{example}

Next we consider the possibility of $J_2$ being disconnected, but $f$ not satisfying the hypotheses of our Theorem~\ref{THM:MAIN}.

\begin{question}
Is there a polynomial skew product $f$ with $J_2$ disconnected, but all $J_z$'s
connected for all $z \in \CC$, such that $H_1(U(f))$ is trivial?  More
generally, is there any endomorphism of $\Pt$ with $J_2$ disconnected, but with
all Fatou components having trivial homology?
\end{question}

By \cite[Proposition 6.6]{JON_SKEW}, in order for $f$ to satisfy the hypotheses of this question, $J_p$ would have to be disconnected.  However, a simple product like $(z,w) \mapsto (z^2-100, w^2)$ does not suffice; note for this map, the basin of attraction of $[1:0:0]$, hence the Fatou set, has nontrivial homology.  Not many examples of non-product polynomial skew products are understood, and the current list of understood examples contains no maps which satisfy the hypotheses of this question.

\subsection{A quadratic family of polynomial skew products}
\label{SEC:QUADRATIC_FAMILY}

We now consider the family of examples $f_a(z,w) = (z^2,w^2+az)$, which are
skew products over $p(z) = z^2$.

The geometry and dynamics in $J_p \times \CC$ were explored in \cite{S}.  For example, there it is established that:
\begin{enumerate}
\item \cite[Theorem 5.1]{S}: $f_a$ is Axiom A if and only if $g_a(w) := w^2+a$ is hyperbolic; and
\item  \cite[Lemma 5.5]{S}: $J_2$ can be described geometrically in the following manner:  $J_{e^{it}}$ is a rotation of angle $t/2$ of $J_{\{z = 1\}}$.  That is, start with $J(g_a)$ in the fiber $J_{\{z = 1\}}$, then as the base point $z=e^{it}$ moves around the unit circle $J_p = S^1$, the corresponding $J_z$'s are rotations of $J(g_a)$ of angle $t/2$, hence the $J_z$'s complete a half turn as $z$ moves once around the base circle.  
\end{enumerate}

Due to the structure of $J_2$, the difference between $f_a$ and the product
$h_a(z,w) = (z^2, w^2+a)$ is one ``twist'' in $J_2$.  In \cite{S} it is shown
that $f_a$ and $h_a$ are in the same hyperbolic component if and only if $a$ is
in the main cardiod of the Mandelbrot set, $\mathcal{M}$.

Note that the extension of $f_a$ to $\Pt$, given by $f_a([Z:W:T]) =
[Z^2,W^2+aZT:T^2]$, is symmetric under the involution $\mathcal{S}([Z:W:T]) = [T:W:Z]$.

\begin{thm}\label{THM:QUADRATIC_FAMILY}
The Fatou set of $f_a$ is the union of the basins of attraction of three
superattracting fixed points: $[0:0:1], [0:1:0]$, and $[1:0:0]$, each of which 
is path-connected.

Moreover:
\begin{itemize}
\item If $a \not \in \mathcal{M}$, then $W^s([0:1:0])$ has infinitely generated first homology.
\item If $a \in \mathcal{M}$, then each of the three basins of attraction
$W^s([0:1:0]), W^s([0:0:1])$ and $W^s([1:0:0])$ is homeomorphic to an open
ball.
\end{itemize}
\end{thm}

\begin{proof}
For any $a$, the fiberwise Julia set $J_0$ is the unit circle $|z| = 1$.
Proposition 4.2 from \cite{ROE_NEWTON} can be modified to show that there is a
local super-stable manifold $W^s_{\rm loc}(J_0)$ that is obtained as the image of a
holomorphic motion of $J_0$ that is parameterized over $\DD_\eps = \{|z| <
\eps\}$, for $\eps > 0$ sufficiently small.   The motion of $(0,w) \in
J_0$ is precisely the connected component of local super-stable manifold of
$(0,w)$ that contains $(0,w)$, which we will call the {\em superstable leaf of
$w$} and denote by $W^s_{\rm loc}(w)$.  By construction, $f_a$ will map the
superstable leaf of $(0,w)$ into the superstable leaf of $(0,w^2) = f_a(0,w)$.
Moreover, the proof of Proposition 4.4 from \cite{ROE_NEWTON} can also be
adapted to show that $W^s_{\rm loc}(J_0)$ is the zero locus
of a pluri-harmonic (hence real-analytic) function.

Pulling back $W^s_{\rm loc}(J_0)$ under iterates of $f_a$, we obtain a global
separatrix $W^s(J_0)$ over the entire unit disc $\DD = \{|z| = 1\}$.  Note that
$W^s(J_0)$ may not be a manifold, since ramification may occur at points where
it intersects the critical locus of $f_a$.  For $|z| < 1$, $J_z$ is the
intersection of $W^s(J_0)$ with $\{z\} \times \CC$ and that $K_z$ is the
intersection of $W^s([0:0:1])\cup W^s(J_0)$ with $\{z\} \times \CC$. 
Thus, any point $(z,w)$ with $|z| < 1$ is in $W^s([0:0:1]) \cup W^s(J_0) \cup W^s([0:1:0])$.

Under the symmetry $\mathcal{S}$, each of the above statements about the super-stable
manifold of $J_0$ corresponds immediately to a statement about the unit circle
$J_\Pi = \{|Z/W| = 1\}$ in the line at infinity $\Pi = \{T=0\}$.  Moreover,
any point in $\Pt$ with $|T| < |Z|$ is in $W^s([1:0:0]) \cup W^s(J_\Pi) \cup
W^s([0:1:0])$.  
Therefore, the Fatou set of $f_a$ is the union of basins of attraction for
three superattracting fixed points $[1:0:0]$, $[0:1:0]$, and $[0:0:1]$.  
Since each of these fixed points is totally invariant, Proposition \ref{W_PATH_CONNECTED}
gives that each of their basins of attraction is path connected.

\vspace{0.05in}

The vertical Julia $J_1$ set over the fixed fiber $z=1$ is precisely the Julia
set of $w \mapsto w^2+a$, which is connected if and only if $a \in
\mathcal{M}$.  In particular, if $a \not \in \mathcal{M}$, it follows from
Theorem \ref{THM:MAIN} that $W^s([0:1:0])$ has infinitely generated first
homology.

\vspace{0.05in}

If $a \in \mathcal{M}$, then, for each $z \in J_p$, $J_z$  is a rotation of the
connected set $J_1$ and Theorem \ref{THM:MAIN}
gives that $W^s([0:1:0])$ is homeomorphic to an open ball.  We will
now use Slodkowski's Theorem on holomorphic motions \cite{SLOD} (see also
\cite[Section 5.2]{HUBB}) to show that $W^s([0:0:1])$ and $W^s([1:0:0]) =
\mathcal{S}(W^s([0:0:1]))$ are homeomorphic to the open bidisc.  

We will extend (in the parameter $z$) the holomorphic motion whose image is
$W^s_{\rm loc}(J_0)$ to a holomorphic motion of $J_0$ parameterized by $z \in
\DD$, having the entire separatrix $W^s(J_0)$ as its image.  Then, by
Slodkowski's Theorem, this holomorphic motion extends (in the fiber $w$) from
$J_0$ to a holomorphic motion of the entire Riemann sphere $\mathbb{P}^1$ that
is also parameterized by $z \in \DD$.  Consequently, $W^s([0:0:1])$ will be the
image of a holomorphic motion of the open disc $\{z=0, |w| < 1\}$,
parameterized by $z \in \DD$.

\vspace{0.05in}
Since $a \in \mathcal{M}$, it also follows from \cite[Proposition
6.4]{JON_SKEW} that for each $z \in \CC$ the fiber-wise critical points 
\begin{eqnarray*}
C_z := \{w \in \CC \,:\,q_z'(w) = 0\}
\end{eqnarray*}
\noindent
 are in $K_z$.  We now check that they are disjoint from
$W^s(J_0)$. 

The union of these fiber-wise critical points is just the horizontal line $w=0$
that stays on one side of $W^s(J_0)$, possibly touching at many points.  Note,
however that they are disjoint at $z=0$.  Consider the point $z_0$ (with $|z_0|
< 1$) of smallest modulus where $w=0$ and $W^s(J_0)$ touch.  Then, there is a
neighborhood of $U$ of $z_0$ in $\CC^2$ in which $W^s(J_0)$ is given by the
zero set of a PSH function $\Psi$.  Changing the sign of $\Psi$ (if necessary)
we can assume that $\Psi \leq 0$ for points in $K_z \cap U$.  The restriction
$\psi(z) =\Psi|_{w=0}$ is a non-positive harmonic function in a neighborhood of
$z_0$ having $\psi(z_0) = 0$, but $\psi(z) < 0$ for $z$ with $|z|< |z_0|$.
This violates the maximum principle.  Therefore, the fiber-wise critical points $C_z$ 
are disjoint from $W^s(J_0)$ for every $z$.

\vspace{0.05in}
Suppose that $\mathcal{D} \subset W^s(J_0)$ is the graph of a
holomorphic function $\nu(z)$ defined on $\{|z| < r\}$, for some $0 < r < 1$.
Then, since $W^s(J_0)$ is disjoint from the horizontal critical locus $w = 0$,
the Implicit Function Theorem gives that $f_a^{-1}(\mathcal{D})$ is
the union of two discs through the pre-images of $\nu(0)$, each given as the graph
of a holomorphic function over $\{|z| < \sqrt{r}\}$.

Let $(0,w) \in J_0$ with preimages $(w_1,0)$ and $(w_2,0)$.  Since
$f_a(W^s_{\rm loc}(w_{1,2})) \subset W^s_{\rm loc}(w)$, the two discs from
$f_a^{-1}(W^s_{\rm loc}(w))$ form extensions of $W^s_{\rm loc}(w_{1})$ and
$W^s_{\rm loc}(w_{2})$, as graphs of holomorphic functions of $|z| <
\sqrt{\eps}$.

Therefore, by taking the preimages under $f_a$, the family of local stable
discs can be extended, each as the graph of a holomorphic function over $|z| <
\sqrt{\eps}$.  Applied iteratively, we can extend them as the graphs of
holomorphic functions over discs $|z| < r$ for any $r < 1$.  In the limit we
obtain global stable curves $W^s(w_0)$ through every $w_0 \in J_0$, each of
which is the graph if a holomorphic function of $z \in \DD$.  Since the global
stable curves of distinct points in $J_0$ are disjoint, their union gives
$W^s(J_0)$ as the image of a holomorphic motion of $J_0$ parameterized by $z
\in \DD$.  \end{proof}


\subsection{Postcritically Finite Holomorphic Endomorphisms}

Until presenting the conjecture of the previous subsection, this paper has been
about endomorphisms with complicated Fatou topology.  The opposite extreme is
that the Fatou topology may also be trivial in many cases.  We suspect one
simple case in which Fatou topology is trivial is when the map is
postcritically finite (PCF).

\begin{question}
Does the Fatou set of a postcritically finite holomorphic endomorphism of $\Pt$ always have trivial homology?
\end{question}

A starting point for investigation into this question could be to attempt to establish it for the postcritically finite examples constructed by Sarah Koch \cite{KochFRphd, KochUSphd}. Heuristic evidence supports that the homology is trivial for Koch's maps.  Her construction provides a class of PCF endomorphisms, containing an infinite number of maps, including the previously studied examples of  \cite{FSpcf} and \cite{CrassPCF}.

\subsection{Other holomorphic endomorphisms of $\Pk$}

As we have demonstrated in Sections~\ref{SEC:ENDO} and~\ref{SEC:SKEW},
given some information about the geometry of the support of $T$, we can apply
the techniques of Sections~\ref{SEC:LINKING} to study the Fatou set of a
holomorphic endomorphism of $\Pt$.  We would like to be able to apply this
theorem to other holomorphic endomorphisms of $\Pk$.  However, specific examples of
holomorphic endomorphisms that are amenable to analytic study are notoriously
difficult to generate.  

One family of endomorphisms which seem a potentially vast area of study are the H\'{e}non-like endomorphisms.  Introduced by Hubbard and Papadapol in \cite{HP2}, and studied a bit further by  Forn{\ae}ss and Sibony in \cite{FSexamples}, these are holomorphic endomorphisms arising from a certain perturbation of the \Henon diffeomorphisms.  The \Henon diffeomorphisms have been deeply studied (e.g., by Bedford Lyubich, and Smillie, \cite{BLS1,BS4}, Bedford and Smillie \cite{BS6, BS9}, Hubbard and Oberste-Vorth \cite{HO1, HO2}, and Forn{\ae}ss and Sibony \cite{FSHenon}).  A natural question which is thus far quite wide open is: how does the dynamics of a \Henon diffeomorphism relate to the dynamics of the perturbed \Henon endomorphism?  Computer evidence suggests the dynamics of H\'{e}non-like endomorphisms is rich and varied.

Specifically concerning the topology of the Fatou set, the main result of
\cite{BS6} is that connectivity of the Julia set is determined by connectivity
of a slice Julia set in a certain unstable manifold.  We ask whether this
result would have implications for the related \Henon endomorphism, which would
allow us to use linking numbers to establish some analog of
Theorem~\ref{THM:MAIN} for \Henon endomorphisms.


%
%

\bibliographystyle{plain}
\bibliography{newton.bib}

\end{document}